\documentclass[12pt]{amsart}
\usepackage[margin=1.5in]{geometry}
\usepackage{amsthm,amsmath,amssymb}
\theoremstyle{plain}
\newtheorem{theorem}{Theorem}[section]
\newtheorem{lemma}[theorem]{Lemma}

\theoremstyle{definition}
\newtheorem{remark}[theorem]{Remark}
\numberwithin{equation}{section}
\newcommand*{\abs}[1]{\lvert#1\rvert}
\newcommand*{\norm}[1]{\lVert#1\rVert}
\newcommand*{\normop}[1]{\lVert#1\rVert_\text{op}}
\newcommand*{\normopU}[1]{\lVert#1\rVert_{U, \text{op}}}
\newcommand*{\defset}[2]{\{\,#1\,\mid\,#2\,\}}

\newcommand*{\CC}{\mathbb C}
\newcommand*{\RR}{\mathbb R}
\DeclareMathOperator{\trace}{trace}
\DeclareMathOperator{\id}{id}
\DeclareMathOperator{\Sym}{Sym}
\DeclareMathOperator{\GL}{GL}
\DeclareMathOperator{\End}{End}
\DeclareMathOperator{\U}{Isom}
\usepackage{xcolor} 
\usepackage{url}
\usepackage[T1]{fontenc}
\usepackage{lmodern}
\usepackage[
  backend=biber,
  style=numeric,
  doi=true,
  url=false,
  giveninits=true,
  isbn=false
]{biblatex}
\usepackage[colorlinks,linkcolor=blue,citecolor=purple]{hyperref}
\title{A note about Jordan's bound on the size of finite linear groups}
\author{Peter M\"uller}
\address{Institute of Mathematics, University of W\"urzburg}
\email{peter.mueller@uni-wuerzburg.de}
\begin{document}
\begin{abstract}
  In 1878 Camille Jordan showed that every finite subgroup
  $G\le\GL_n(\CC)$ has an abelian normal subgroup $A$ such that
  $\abs{G/A}$ is bounded in terms of $n$, but he did not give an
  explicit bound. An explicit bound was obtained by Blichfeldt in a
  series of papers beginning in 1904, using representation-theoretic
  methods. In 1911 Bieberbach gave a geometric proof, which is quite
  different from the approaches of Jordan and Blichfeldt, together
  with an explicit bound. Frobenius simplified this proof in the same
  year, and the resulting argument is still the simplest known. We
  present a self-contained and streamlined variant of Frobenius's
  argument, yielding the bound $\abs{G/A}\le25^{n^2}$.
\end{abstract}
\maketitle
\section{Introduction}
A classical topic in group theory is the study of finite subgroups of
$\GL_n(\CC)$. Fixing $n$, such groups can be arbitrarily large. For
instance, for every $m\ge1$ the group of scalar matrices
$\defset{\zeta I_n}{\zeta\in\CC, \zeta^m=1}$ has order $m$. However,
large finite subgroups of $\GL_n(\CC)$ have large abelian normal
subgroups by the following classical result of Camille Jordan:
\begin{theorem}[Jordan 1878]
  Let $G\le\GL_n(\CC)$ be a finite group. Then $G$ has an abelian
  normal subgroup $A$ such that $\abs{G/A}\le f(n)$, where $f(n)$
  depends only on $n$.
\end{theorem}
The statement of this theorem appears on page 91, and its proof is
given in Chapitre II of Jordan's long paper
\cite{Jordan:GLnC}. Jordan's proof, which is purely algebraic, does
not give an explicit bound $f(n)$. See
\cite{Breuillard:OriginalJordan} for a modern exposition of Jordan's
original argument, and many details about the history of Jordan's
theorem.

Beginning in 1904, several papers by Blichfeldt used representation
theoretic methods to obtain a different proof and the explicit bound
$f(n)=n!6^{(n-1)(\pi(n+1)+1)}$, where $\pi(n+1)$ is the number of
primes less than or equal to $n+1$. See \cite[\S30]{Dornhoff:A} for
this bound, or \cite[Theorem (14.12)]{Isaacs} and \cite[\S70]{Speiser}
for a slightly weaker bound, where $6$ is replaced by $12$.

Later in 1911 Bieberbach \cite{Bieberbach:Jordan} found a completely
different proof of Jordan's theorem, along with the bound
$f(n)=(1+32^4n^{10})^{2n^2}$. In the same year Frobenius
\cite{Frobenius:Jordan} simplified Bieberbach's proof and obtained the
better bound $f(n)=(\sqrt{8n}+1)^{2n^2}-(\sqrt{8n}-1)^{2n^2}$.

A modern account of Frobenius's method is Dixon's version \cite[Theorem
5.7]{Dixon:LinearGroups}, yielding the bound $f(n)=(49n)^{n^2}$ (which
is somewhat weaker than the previous bounds).

In \cite{Tao:Gromov} Tao sketched a variant of Frobenius' method. Full
details are given in \cite[Section 2]{BreuillardGreen:III}. The bound
$f(n)=n^{Cn^3}$ for some undetermined absolute constant $C$ is worse
than the previous bounds.

These proofs yield no information regarding the optimal bound
$f(n)$ for a general $n$. To date, Blichfeldt's bound is
asymptotically the best known one which can be proven without invoking
the classification of the finite simple groups. Note that if $c>6$,
then $f(n)=c^{n^2/\log n}$ works for all sufficiently large $n$ by the
prime number theorem.

In Jordan's theorem we necessarily have $f(n)\ge (n+1)!$ for $n\ge4$,
as can be seen as follows: Let the symmetric group $\Sym(n+1)$ act on
$\CC^{n+1}$ by permuting the coordinates. The $n$-dimensional subspace
which consists of those vectors whose coordinates sum up to $0$ is
invariant under $\Sym(n+1)$. This yields an embedding of $\Sym(n+1)$
into $\GL_n(\CC)$. Since $\Sym(n+1)$ does not have a nontrivial
abelian normal subgroup for $n\ge4$, it follows that $f(n)\ge (n+1)!$.

Conversely, using the classification of the finite simple groups,
Collins \cite{Collins:Jordan0} showed that Jordan's theorem in fact
holds with $f(n)=(n+1)!$ once $n\ge71$. He also determined the optimal
cases for all $n\le70$.

While more than 100 years old, even today Frobenius' proof is
considered to be the simplest approach to Jordan's theorem. His proof
(as well as Bieberbach's precursor) consists of three steps:
\begin{itemize}
\item[Step 1:] The unitary trick -- known already to Bieberbach and
  Frobenius -- allows us to assume that $G$ is a subgroup of the
  unitary group $\operatorname{U}_n(\CC)$.
\item[Step 2:] Using a suitable matrix norm $\norm{\cdot}$ on
  $\CC^{n\times n}$ and a sufficiently small positive constant $c$,
  one defines the subgroup $A$ which is generated by all $g\in G$ with
  $\norm{I_n-g}<c$ and shows that it is an abelian normal subgroup of
  $G$.
\item[Step 3:] A simple volume packing argument or alternatively the
  pigeon-hole principle shows that $\abs{G/A}$ is bounded from above
  in terms of $n$ and $c$ only.
\end{itemize}
The crucial piece for Step 2 is Lemma \ref{l:xy=yx} (or variants
thereof) below. Its proof by Frobenius or the one in the standard
references \cite[Proof of Th.\ 5.7, part (a)]{Dixon:LinearGroups} and
\cite[(36.17) Lemma]{CurtisReiner:Rep} is somewhat involved. Below we
provide a different proof which we believe is more natural.

Furthermore, we replace Step 1 by a weaker result which does not
require the spectral theorem, and which is sufficient for our
purposes.

We prove Jordan's theorem with the bound $f(n)=25^{n^2}$.
\section{Normed vector spaces and linear isometries}
Let $V$ be a finite-dimensional normed vector space over $\CC$. Let
$\norm{v}$ denote the norm of $v\in V$. Thus $\norm{\cdot}$ is a
function from $V$ to $\RR$ such that for each $v, w\in V$ and
$\lambda\in\CC$ the following holds:
\begin{align*}
  \norm{v}&\ge0,\\
  \norm{v}&=0\text{ if and only if }v=0,\\
  \norm{\lambda v}&=\abs{\lambda}\cdot\norm{v},\\
  \norm{v+w}&\le\norm{v}+\norm{w}.                           
\end{align*}
A linear isometry of $V$ is a map $x\in\GL(V)$ such that
$\norm{xv}=\norm{v}$ for all $v\in V$. Let $\U(V)$ denote the group of
linear isometries of $V$.
\begin{lemma}\label{l:uv}
  Let $V$ be a finite-dimensional complex vector space and
  $G\le\GL(V)$ be a finite group. Then $V$ can be equipped with a norm
  such that $G\le \U(V)$, that is, $G$ acts by linear isometries on
  $V$.
\end{lemma}
\begin{proof}
  Let $\norm{\cdot}'$ be an arbitrary norm on $V$. Set
  \[
    \norm{v}=\sum_{h\in G}\norm{hv}'
  \]
  for $v\in V$. Then $\norm{\cdot}$ is a norm on $V$ and
  \[
    \norm{gv}=\sum_{h\in G}\norm{h(gv)}'=
           \sum_{h\in G}\norm{(hg)v}'=\norm{v},
  \]
  because if $h$ runs through $G$, then so does $hg$.
\end{proof}
We recall the definition of the operator norm of $a\in\End(V)$. Since
$V$ is finite-dimensional, the set $\defset{v\in V}{\norm{v}=1}$ is
compact. Hence the continuous function $V\to\RR$, $v\mapsto\norm{av}$
assumes a maximum on this set. Accordingly, the operator norm of $a$
is defined as
$\normop{a}=\max\defset{\norm{av}}{\norm{v}=1}$. Immediate
consequences of this definition are
\begin{align*}
  \norm{av}&\le\normop{a}\cdot\norm{v},\\
  \normop{\lambda a}&=\abs{\lambda}\cdot\normop{a},\\
  \normop{a+b}&\le\normop{a}+\normop{b},\\
  \normop{ab}&\le\normop{a}\cdot\normop{b}
\end{align*}              
for all $a, b\in\End(V)$, $v\in V$, and $\lambda\in\CC$.

We note three more immediate consequences of the definition of the
operator norm:
\begin{align*}
  \normop{xay} &= \normop{a} \text{ for all $a\in\End(V)$ and $x,
                 y\in\U(V)$},\\
  \normop{x} &= 1 \text{ for all $x\in\U(V)$},\\
  \abs{\mu} &\le \normop{a} \text{ for each eigenvalue $\mu$ of
              $a\in\End(V)$}.
\end{align*}

\section{Commuting elements}
The key step in the proof of Jordan's theorem is the following lemma,
which shows that elements in a finite subgroup of $\U(V)$ sufficiently
close to the identity commute.
\begin{lemma}\label{l:xy=yx}
  Let $V$ be a finite-dimensional complex vector space with norm
  $\norm{\cdot}$ and $\id_V$ be the identity map on $V$. Then the
  following holds:
  \begin{itemize}
  \item[(a)] If $\normop{\id_V-x}<1$ for $x\in\End(V)$, then $\trace
    x\ne0$.
  \item[(b)] Suppose that $x,y\in\U(V)$ generate a finite group. If
    $\normop{\id_V-x}<1/2$ and $\normop{\id_V-y}<1/2$, then $xy=yx$.
  \end{itemize}
\end{lemma}
\begin{proof}
  \begin{itemize}
  \item[(a)] If $\lambda$ is an eigenvalue of $x$, then $1-\lambda$ is
    an eigenvalue of $\id_V-x$, hence
    $\abs{1-\lambda}\le\normop{\id_V-x}<1$.

    Let $\lambda_1,\dots,\lambda_n$ be the eigenvalues of $x$,
    where $n=\dim V\ge1$. Then
    \[
      n %
      > \sum\abs{1-\lambda_i}%
       \ge \abs{\sum(1-\lambda_i)}%
      =\abs{n-\trace x},
    \]
    hence $\trace x\ne0$.
  \item[(b)] We prove the claim by induction on $\dim V$. The base
    case $\dim V=1$ is immediate. Also, clearly $xy=yx$ if $x$ is a
    scalar map. Thus suppose that $x$ is not a scalar map. Let $z$ be
    a non-scalar element of the finite group generated by $x$ and $y$
    for which $\normop{\id_V-z}$ is minimal. From
  \begin{align*}
    \id_V-x^{-1}z^{-1}xz%
    &=x^{-1}z^{-1}(zx-xz)\\
    &=x^{-1}z^{-1}\left((\id_V-z)(\id_V-x)-(\id_V-x)(\id_V-z)\right)
  \end{align*}
  we obtain
  \begin{align*}
    \normop{\id_V-x^{-1}z^{-1}xz}%
    &\le \normop{(\id_V-z)(\id_V-x)-(\id_V-x)(\id_V-z)}\\
    &\le2 \cdot\normop{\id_V-x}\cdot\normop{\id_V-z}\\
    &< \normop{\id_V-z}.
  \end{align*}
  By the minimal choice of $z$ we see that $x^{-1}z^{-1}xz$ is a
  scalar map $\lambda\cdot\id_V$, hence
  \[
    z^{-1}xz=\lambda x
  \]
  for some $\lambda\in\CC$.
  
  Taking the trace gives $\lambda=1$ because $\trace x\ne0$ by (a).

  So $z$ commutes with $x$, and for the same reason $z$ commutes with
  $y$.

  As $z$ has finite order, $V$ is a direct sum of the eigenspaces of
  $z$. These eigenspaces are proper subspaces of $V$, because $z$ is
  not a scalar map. Moreover, as the elements $x$ and $y$ commute with
  $z$, they map each eigenspace of $z$ to itself. Thus in order to
  show that $x$ and $y$ commute, we need only show that the
  restrictions $x_U$ and $y_U$ of $x$ and $y$ to $U$ commute for each
  eigenspace $U$ of $z$.

  Thus the lemma follows by induction on $\dim V$ once we verify that
  $x_U$ and $y_U$ satisfy the assumptions of the lemma.

  Of course, $x_U$ and $y_U$ generate a finite group.

  Let $\norm{\cdot}_U$ be the restriction of $\norm{\cdot}$ to $U$,
  and ${\normopU{\cdot}}$ be the corresponding operator norm on
  $\End(U)$.

  Then $x_U, y_U\in\U(U)$ with respect to
  $\norm{\cdot}_U$. Furthermore, the definition of the operator norm
  shows that $\normopU{\id_U-a_U}\le\normop{\id_V-a}$ for each
  $a\in\End(V)$ which maps $U$ to itself. In particular,
  $\normopU{\id_U-x_U}<1/2$ and $\normopU{\id_U-y_U}<1/2$.

  Thus all the assumptions of the lemma are satisfied for $x_U$ and
  $y_U$, so $x_U$ and $y_U$ commute by the induction hypothesis.
\end{itemize}
\end{proof}
\section{The ball packing argument}

Set $V=\CC^n$ and let $G\le\GL(V)$ be a finite group. Via Lemma
\ref{l:uv} we equip $V$ with a norm such that $G\le\U(V)$. Let
$\normop{\cdot}$ be the corresponding operator norm on $\End(V)$.

We let
\[
  B(r)=\defset{a\in\End(V)}{\normop{a}<r}
\]
be the open ball of radius $r$ around the zero map in $\End(V)$ with
respect to the operator norm.

Set $M=\defset{x\in G}{\normop{\id_V-x}<1/2}$. As
\[
  \normop{\id_V-y^{-1}xy}%
=\normop{y^{-1}\cdot(\id_V-x)\cdot y}%
=\normop{\id_V-x}
\]
for all $x\in M$ and $y\in G$ we see that $M$ is invariant under
conjugation by elements of $G$. Furthermore, by Lemma \ref{l:xy=yx}
any two elements of $M$ commute. Thus the group $A$ generated by $M$
is an abelian normal subgroup of $G$.

Let $x_1,\dots,x_m$ be representatives of the left cosets of $A$ in
$G$.

We claim that the $m$ balls $x_i+B(1/4)$ are pairwise
disjoint. Suppose that there exists $y\in\End(V)$ contained in both
$x_i+B(1/4)$ and $x_j+B(1/4)$. Then
\[
  \normop{x_i-x_j}=\normop{(x_i-y)+(y-x_j)}\le\normop{x_i-y}+\normop{y-x_j}
  < 1/4+1/4 = 1/2,
\]
and therefore $\normop{x_i\cdot x_j^{-1}-\id_V}<1/2$, so
$x_i\cdot x_j^{-1}\in M\subseteq A$, thus $x_i=x_j$.

As $\normop{x_i}=1$, we obtain $x_i+B(1/4)\subseteq B(1+1/4)$ for all
$i$. Thus the ball $B(5/4)$ of radius $5/4$ contains $m$ disjoint
balls of radius $1/4$.

As a real vector space, $\End(V)$ is isomorphic to $\RR^{2n^2}$ and
therefore carries a Lebesgue measure $\lambda$. Since $B(1/4)$ is an
open and nonempty subset of $\End(V)$, it has positive Lebesgue
measure, say $\beta=\lambda(B(1/4))$. By the translation invariance of
the Lebesgue measure we have $\lambda(x_i+B(1/4))=\beta$. Furthermore,
since $B(5/4)=5B(1/4)$, the scaling property of Lebesgue measure
implies $\lambda(B(5/4))=5^{2n^2}\lambda(B(1/4))=25^{n^2}\beta$.

Thus the ball $B(5/4)$, which has Lebesgue measure $25^{n^2}\beta$,
contains $m$ pairwise disjoint balls of Lebesgue measure $\beta$.
Since $\beta>0$, it follows that $\abs{G/A}=m\le 25^{n^2}$, as
required.
\begin{remark}
  One actually gets a slightly better bound, because $\norm{x_i}=1$
  implies $x_i+B(1/4)\subseteq B(1+1/4)\setminus B(1-1/4)$ for all
  $i$, hence $\abs{G/A}=m\le 25^{n^2}-9^{n^2}$ by comparing volumes as
  above.
\end{remark}
\printbibliography
\end{document}